\documentclass[leqno,11pt]{amsart}
\usepackage[T1]{fontenc}
\usepackage[utf8]{inputenc}
\usepackage{amsmath}
\usepackage{amssymb,amscd}
\usepackage{amsthm}
\usepackage{latexsym}
\usepackage{stmaryrd}
\usepackage{enumerate}
\usepackage[myheadings]{fullpage}
\usepackage{array}
\usepackage{color}  

\usepackage[linktoc=all, hyperindex]{hyperref}
\hypersetup{colorlinks,  citecolor=blue,  filecolor=blue, linkcolor=red, urlcolor=blue}

\newtheorem{thm}{Theorem}[section]
\newtheorem{lem}[thm]{Lemma}
\newtheorem{prop}[thm]{Proposition}
\newtheorem{cor}[thm]{Corollary}

\newtheorem{theoremx}{Theorem}

\theoremstyle{definition}
\newtheorem{definition}[thm]{Definition}
\newtheorem{rem}[thm]{Remark}
\newtheorem{ex}[thm]{Example}

\newcommand{\p}{\mathfrak p}

\newcommand{\kk}{\Bbbk}

\def\opn#1#2{\def#1{\operatorname{#2}}} 
\opn\tr{tr}
\opn\GL{GL}
\opn\SL{SL}
\opn\spec{Spec}
\opn\depth{depth}
\opn\height{ht}
\opn\res{res}
\opn\ord{ord}
\opn{\chara}{char}

\title{Nearly Gorenstein cyclic quotient singularities}

\author{Alessio Caminata}
\address{Alessio Caminata, Institut de Math\'{e}matiques, Universit\'{e} de Neuch\^{a}tel, Rue Emile-Argand 11, CH-2000 Neuch\^{a}tel, Switzerland}
\email{alessio.caminata@unine.ch}

\author{Francesco Strazzanti}
\address{Francesco Strazzanti, Dipartimento di Matematica, Universit\`a di Bologna, Piazza di Porta San Donato 5, 40126 Bologna, Italy}
\email{francesco.strazzanti@unibo.it}

\thanks{The second author was partially supported by INdAM, more precisely he was ``titolare di una borsa per l'estero dell'Istituto Nazionale di Alta Matematica'' and ``titolare di un Assegno di Ricerca dell'Istituto Nazionale di Alta Matematica''.}


\begin{document}

\begin{abstract}
We investigate the nearly Gorenstein property among $d$-dimensional cyclic quotient singularities $\kk\llbracket x_1,\dots,x_d\rrbracket^G$, where $\kk$ is an algebraically closed field and $G\subseteq\GL(d,\kk)$ is a finite small cyclic group whose order is invertible in $\kk$. We prove a necessary and sufficient condition to be nearly Gorenstein that also allows us to find several new classes of such rings.
\end{abstract}

\maketitle

\section*{Introduction}

Gorenstein rings are among the most important objects in commutative algebra and appear in several contexts.
On the other hand, despite their celebrated \textit{ubiquity} \cite{Bas63, Hun99}, the class of Gorenstein rings is not so large, compared for instance with that of Cohen-Macaulay rings.
In many significant cases one encounters Cohen-Macaulay rings which are not Gorenstein rings. 
For this reason, many researchers started looking for generalizations of the notion of Gorenstein aiming to find a class of Cohen-Macaulay rings which is still able to capture some of the interesting properties of Gorenstein rings.
In recent years, two of these new classes of rings have drawn particular attention: \textit{almost Gorenstein} and \textit{nearly Gorenstein rings}.
Almost Gorenstein rings were first defined by Barucci and Fr\"oberg \cite{BF97} for one-dimensional analytically unramified rings and later generalized by Goto~et~al. \cite{GMP13, GTT15}.
This notion has already been largely investigated and many properties are known, see for example \cite{EGI19, GTT16, HJS19,T18} and the references therein.
On the other hand, nearly Gorenstein rings, which are the object of interest of this paper, have been introduced even more recently by Herzog, Hibi, and Stamate \cite{HHS19} in 2019, although their defining property was already examined by Ding \cite{Din93}, Huneke and Vraciu \cite{HV06}, and Striuli and Vraciu \cite{SV11}. Moreover, nearly Gorenstein rings have been studied in several contexts, such as zero-dimensional schemes \cite{KLL19}, affine semigroup rings \cite{HJS19}, and affine monomial curves \cite{MS20}. See also \cite{EGI19,DKT20,Kob20, Kum20, Rah20} for other related results.

\par To explain the definition and the motivation of nearly Gorenstein rings we start with a Cohen-Macaulay local ring $(R,\mathfrak{m})$ which admits a canonical module $\omega_R$.
The trace of the canonical module, denoted by $\tr(\omega_R)$, is defined as the sum of the ideals $\varphi(\omega_R)$, where the sum is taken over all the $R$-module homomorphisms $\varphi:\omega_R\rightarrow R$.
The importance of $\tr(\omega_R)$ comes from the fact that it describes the non-Gorenstein locus of $R$, since the localization $R_\p$ at a prime ideal $\p$ is not Gorenstein if and only if $\tr(\omega_R) \subseteq \p$.
In particular, it follows that $R$ is Gorenstein if and only if $\tr(\omega_R)=R$. 
For this reason, one defines $R$ to be nearly Gorenstein when $\mathfrak{m} \subseteq \tr(\omega_R)$. It is now clear that a nearly Gorenstein ring is Gorenstein on the punctured spectrum, but the converse does not occur in general. Moreover, it also holds that a one-dimensional almost Gorenstein ring is nearly Gorenstein, even though this is no longer true in the higher dimensional case, where the relation between these two notions remains unclear.

\par In this paper, we look at the nearly Gorenstein property for quotient singularities.
Let $R=\kk\llbracket x_1,\dots,x_d\rrbracket$ be a $d$-dimensional formal power series ring over an algebraically closed field $\kk$ and let $G$ be a finite subgroup of $\GL(d,\kk)$ acting linearly on $R$.
The corresponding invariant ring $R^G$ is the completion at the origin of the coordinate ring of the quotient variety $\mathbb{A}_{\kk}^d/G$, so we will refer to it as a \emph{quotient singularity}.
The study of these objects and their properties lies at the intersection of several branches of mathematics and has been largely explored both from a geometric and an algebraic point of view.
In the modular case, i.e., when the characteristic of $\kk$ divides the order of the group $G$, even the Cohen-Macaulay property is not fully understood (see e.g. \cite[Example~8.0.9]{CW11} or \cite{Kem99}), so we will rather focus on the non-modular situation, that is when $\chara\kk\nmid|G|$.
Under this assumption it is well known that the invariant ring $R^G$ is a complete local normal domain and it is Cohen-Macaulay thanks to Hochster-Eagon's Theorem \cite{HE71}.
Moreover, thanks to an old result of Prill \cite{Pri67} it is not restrictive to assume further that the acting group is small, i.e., it does not contain pseudo-reflections. 
In this case, by a result of Watanabe \cite{Wat74a, Wat74b} the Gorenstein property of these rings is also well understood. Namely, $R^G$ is Gorenstein if and only if the group $G$ is contained in $\SL(d,\kk)$.
Therefore, it arises as a natural problem to look for a characterization of the nearly Gorenstein property for these rings. In fact, we investigate precisely this question for an important class of quotient singularities: cyclic quotient singularities, i.e., when the group $G$ is cyclic.

\par For this class, we are able to find a numerical criterion which gives a necessary and sufficient condition for the ring $R^G$ to be nearly Gorenstein.
Using this criterion, we identify several families of nearly Gorenstein rings.
We recall that if $G$ is a cyclic small subgroup of $\GL(d,\kk)$ of order $n$ with $\chara\kk\nmid n$, we can assume that it is generated by a diagonal matrix $\phi=\mathrm{diag}(\lambda^{t_1},\dots,\lambda^{t_d})$, where $\lambda$ is a primitive $n$-th root of unity in $\kk$ and $t_1,\dots,t_d$ are positive integers such that $\gcd(t_{i_1}, \dots, t_{i_{d-1}},n)=1$ for every $(d-1)$-tuple with distinct integers $i_1, \dots, i_{d-1} \in \{1, \dots, d\}$.
We denote the corresponding invariant ring $R^G$ by $\frac{1}{n}(t_1,\dots,t_d)$.

\begin{theoremx}[see Proposition~\ref{prop:cyclicNGfamily} and Corollaries~\ref{cor:cyclicarenearlyG}, \ref{cor:nleq3isNG}, \ref{cor:VeroneseareNG}]\label{thm:introductioncyclic}
Let $n,d\geq2$ and $t_1,\dots,t_d\geq1$ be integers and assume that at least one of the following holds:
\begin{itemize}
	\item $d=2$;
	\item $n\leq 3$;
	\item $t_1\equiv\cdots\equiv t_d\equiv1\mod n$;
	\item $t_1\equiv\cdots\equiv t_{d-1}\equiv1\mod n$ and $t_d\equiv -d+2\mod n$.
\end{itemize}
Then, the cyclic quotient singularity $\frac{1}{n}(t_1,\dots,t_d)$ is nearly Gorenstein.
\end{theoremx}

In the case $t_1\equiv\cdots\equiv t_d\equiv1\mod n$ the corresponding invariant ring is a Veronese subalgebra of $R$. We also notice that, when $R$ is a Gorenstein positively graded $\kk$-algebra with positive dimension, Veronese subalgebras of $R$ are known to be nearly Gorenstein by \cite[Corollary 4.7]{HHS19}. 
Moreover, Theorem~\ref{thm:introductioncyclic} says that if the dimension is two or if the order of the group is at most $3$, then cyclic quotient singularities are always nearly Gorenstein. However, as soon as these assumptions are dropped we may find examples of cyclic quotient singularities that are not nearly Gorenstein. For instance, the invariant ring $\frac{1}{4}(1,2,3)$ is not nearly Gorenstein (see Example~\ref{ex:cyclicnotNG}).
More generally, the numerical criterion we proved can be implemented to find all nearly Gorenstein cyclic quotient singularities for some values of $n$ and $d$. For example, see Table \ref{Table NG} for an exhaustive list of non-isomorphic nearly Gorenstein cyclic quotient singularities with small values of $n$ and $d$.

\par In order to measure the distance of a quotient singularity $R^G$ to be Gorenstein or nearly Gorenstein, one can consider its \emph{residue} which is the length $\res(R^G)=\ell_{R^G} (R^G/\tr(\omega_{R^G}))$. We have that $\res(R^G)=0$ if and only if $R^G$ is Gorenstein, and $\res(R^G)=1$ precisely when $R^G$ is nearly Gorenstein, but not Gorenstein.
Thus, from Theorem~\ref{thm:introductioncyclic} follows that every two-dimensional cyclic quotient singularity $R^G$ has $\res(R^G)\leq1$.
However, already in dimension 3 we are able to produce cyclic quotient singularities of arbitrarily large residue.

\begin{theoremx}[see Theorem~\ref{theorem:residuelarge}]\label{theorem:residueintroduction}
	Let $n$ and $m$ be two coprime positive integers with $n\geq3$ and $m<\lceil\frac{n}{2} \rceil$. Then the cyclic quotient singularity $\frac{1}{n}(1,m,n-1)$ has residue $m$.
\end{theoremx}

\par In Section 3 we consider the field of complex numbers $\mathbb{C}$ and we turn our attention to the two-dimensional case, where the finite small group $G\subseteq\GL(2,\mathbb{C})$ is not necessarily cyclic.
The nearly Gorenstein property of the corresponding invariant rings was studied by Ding, who gave a complete classification of nearly Gorenstein two-dimensional quotient singularities (see \cite[Proposition~3.5]{Din93}).
However, somehow a case was left out of Ding's classification: it is the invariant ring of the octahedral group $\mathbb{O}_{11}$ obtained by adding to the binary octahedral subgroup of $\SL(2,\mathbb{C})$ a cyclic generator of the form $\mathrm{diag}(\lambda,\lambda)$, where $\lambda\in\mathbb{C}$ is a primitive root of unity of order $22$. We prove that this ring is nearly Gorenstein in Proposition~\ref{prop:O11isNG}.

\par The structure of the paper is the following.
First, in Section~\ref{section:preliminaries} we review some basic definitions and notations on nearly Gorenstein rings and quotient singularities.
Then, in Section~\ref{section:cyclic} we focus on nearly Gorenstein cyclic quotient singularities. In Theorem~\ref{cyclic nearly Gorenstein} we prove a numerical criterion that characterizes them and we use this to provide several classes of nearly Gorenstein rings as stated in Theorem~\ref{thm:introductioncyclic}.
Finally, in Section~\ref{section:dimension2} we study the nearly Gorenstein octahedral singularity mentioned above which completes Ding's classification.

\subsection*{Acknowledgments}
This work began at the Institute of Mathematics of the University of Barcelona (IMUB). The authors would like to express their gratitude to the IMUB for providing a fruitful work environment.

\section{Preliminaries}\label{section:preliminaries}
In this section we recall some basic definitions and standard facts on nearly Gorenstein rings and quotient singularities.

\subsection{Nearly Gorenstein rings}
Let $(R,\mathfrak{m})$ be a Cohen-Macaulay local ring which admits a canonical module $\omega_R$. The trace of the canonical module, denoted by $\tr(\omega_R)$, is the sum of the ideals $\varphi(\omega_R)$ for any $R$-module homomorphism $\varphi:\omega_R\rightarrow R$. In other words, we have
\[
\tr(\omega_R)=\sum_{\varphi\in\mathrm{Hom}_R(\omega_R,R)}\varphi(\omega_R).
\]
The trace of $\omega_R$ describes the non-Gorenstein locus of $R$. In fact, given a prime ideal $\mathfrak{p}\subseteq R$, then $R_{\mathfrak{p}}$ is not Gorenstein if and only if $\tr(\omega_R)\subseteq\mathfrak{p}$ (cf. \cite[Lemma~2.1]{HHS19}).
In particular, since $\tr(\omega_R)$ is an ideal, one has that $R$ is Gorenstein if and only if $\tr(\omega_R)=R$.
\begin{definition}[Herzog, Hibi, Stamate \cite{HHS19}]
	$R$ is called \emph{nearly Gorenstein} if $\mathfrak{m}\subseteq \tr(\omega_R)$.
\end{definition}
It is immediately clear from the definition that Gorenstein rings are nearly Gorenstein and that $R$ is nearly Gorenstein but not Gorenstein if and only if $\tr(\omega_R)=\mathfrak{m}$.
In order to give a measure to the distance of a ring to be Gorenstein or nearly Gorenstein, one defines the \emph{residue} of $R$ as
\[
\res(R)=\ell_R (R/\tr(\omega_R))\in\mathbb{N}\cup\{\infty\}.
\]
The ring $R$ is Gorenstein if and only if $\res(R)=0$ and it is nearly Gorenstein if and only if $\res(R)\leq 1$.

If there exists a canonical module $\omega_R$ that is also an ideal of $R$ we say that $\omega_R$ is a canonical ideal of $R$. In this case there is a useful formula to find its trace. We denote the total ring of fractions of $R$ by $Q(R)$.

\begin{lem}\label{lemma:tracecanonicalideal} \cite[Lemma 1.1]{HHS19}
Let $(R, \mathfrak{m})$ be a local domain with a canonical ideal $\omega_R$. Then, the trace ideal of the canonical module of $R$ is equal to $\tr(\omega_R)=\omega_R (R:_{Q(R)}\omega_R)$.
\end{lem}

In particular, if $\omega_R$ is a canonical ideal, then it is included in $\tr(\omega_R)$ because $1\in( R:_{Q(R)}\omega_R)$.

\subsection{Quotient singularities}
Let $\kk$ be an algebraically closed field and let $G$ be a finite subgroup of $\GL(d,\kk)$ such that the order $|G|$ of $G$ is coprime with the characteristic of  $\kk$.
We consider a power series ring $R=\kk\llbracket x_1,\dots,x_d\rrbracket$ over $\kk$. 
The group $G$ acts linearly on $R$ with the action on the variables $x_1,\dots,x_d$ given by matrix multiplication.  We denote by $R^G$ the ring of invariants under this action and we will call it also  \emph{(non-modular) quotient singularity}.

We recall that an element $\sigma\in G$ is called pseudo-reflection if the fixed subspace $\{v\in\kk^d:\ \sigma v=v\}$ has dimension $d-1$. 
We will always assume that the acting group $G$ is \emph{small}, i.e., that it does not contain pseudo-reflections. This is not restrictive in our setting. In fact, by a theorem of Prill \cite{Pri67} if $G$ is not small we can replace $R$ by another power series ring $S$ and find a small finite linear group $H$ such that $R^G\cong S^H$. This is essentially a consequence of the Chevalley–Shephard–Todd Theorem which implies that the ring of invariants of a finite group generated by pseudo-reflections acting on a power series ring is again a regular local ring. 

Under the previous assumptions, the quotient singularity $R^G$ is a Cohen-Macaulay complete local normal domain of dimension $d$. Moreover, Watanabe \cite{Wat74a, Wat74b} proved that $R^G$ is Gorenstein if and only if $G \subseteq \SL(d,\kk)$. In this case $R^G$ is called \emph{special quotient singularity}. If $G$ is a cyclic group, then $R^G$ is called \emph{cyclic quotient singularity}. The \emph{Kleinian singularities} are the two-dimensional complex special quotient singularities $\mathbb{C}\llbracket x_1,x_2\rrbracket^G$.

In order to study the nearly Gorenstein property of quotient singularities, it is important to understand their canonical module. To this purpose, we introduce the following definition.

\begin{definition}
Let $R=\kk\llbracket x_1,\dots,x_d\rrbracket$ be a power series ring and let $G$ be a finite small subgroup of $\GL(d,\kk)$ such that $\chara\kk\nmid|G|$. We say that an element $f\in R$ is a \emph{$G$-canonical element of $R$} if $\sigma(f)=\det\sigma\cdot f$ for all $\sigma\in G$.
\end{definition}

The previous definition is motivated by the following result due to Singh \cite{Sin70} and Watanabe \cite[Theorem $1'$]{Wat74b}  (see also \cite{Sta78} or \cite[Theorem~6.4.9]{BH98} for an alternative proof).

\begin{thm}[Singh-Watanabe]\label{thm:singh-watanabe}
	Let $R$ and $G$ be as above and let $f\in R$ be a $G$-canonical element, then $fR\cap R^G$ is a canonical ideal of $R^G$.
\end{thm}

\section{Cyclic quotient singularities}\label{section:cyclic}
In this section we focus on the nearly Gorenstein property for cyclic quotient singularities.
We consider a formal power series ring $R=\kk\llbracket x_1, \dots, x_d\rrbracket$ over an algebraically closed field $\kk$ and a finite small cyclic group $G \subseteq \GL(d,\kk)$ such that $|G|=n$ is not zero in $\kk$.

Since $G$ is a finite cyclic group, we can assume that it is generated by a diagonal matrix $\phi$ of the form
\[
\phi =\mathrm{diag}(\lambda^{t_1}, \dots, \lambda^{t_{d}})= \left(\begin{matrix}
\lambda^{t_1} & 0 & \dots & 0 \\
0 & \lambda^{t_2} & \dots & 0 \\
\vdots & \vdots & \ddots & \vdots \\
0 & 0 & \dots & \lambda^{t_d} \\
\end{matrix} \right) 
\]

\vspace{2mm}

\noindent where $\lambda$ is a primitive $n$-th root of unit in $\kk$ and $t_1, \dots, t_d \geq 0$ are integers. 
If $t_d\equiv0\mod n$, one can set $S=\kk\llbracket x_1, \dots, x_{d-1}\rrbracket$ and $H\subseteq\GL(d-1,\kk)$ the group generated by $\mathrm{diag}(\lambda^{t_1}, \dots, \lambda^{t_{d-1}})$, then $R^G$ is nearly Gorenstein if and only if $S^{H}$ is Gorenstein, by \cite[Proposition 4.5]{HHS19}. For this reason, we will assume without loss of generality that $t_1, \dots, t_d \not\equiv 0\mod n$.
In this case, the lack of pseudo-reflections in $G$ is equivalent to the condition $\gcd(t_{i_1}, \dots, t_{i_{d-1}},n)=1$ for every $(d-1)$-tuple with distinct integers $i_1, \dots, i_{d-1} \in \{1, \dots, d\}$.
With these conventions, we denote the cyclic quotient singularity $R^G$ by $\frac{1}{n}(t_1,\dots,t_d)$. We point out that this notation is not unique. For instance $\frac{1}{3}(1,1,2)=\frac{1}{3}(2,2,1)$ are equal because they are invariant rings with respect to the same group.
Since the action of $G$ on $R$ is diagonal, the $\kk$-algebra $\frac{1}{n}(t_1,\dots,t_d)$ can be generated by monomials, more precisely one can choose a (non-minimal) system of generators as follows
\[
R^G=\kk\llbracket x_1^{\alpha_1}\dots x_d^{\alpha_d} \mid \alpha_1+\dots+ \alpha_d \leq n \text{\ and \ } \alpha_1t_1+\dots +\alpha_dt_d \equiv 0 \mod n\rrbracket.
\]

\begin{prop}\label{prop:cyclictrace}
	Let $R$ and $G$ be as above. A $G$-canonical element of $R$ is given by $f=x_1 x_2 \dots x_d$. Moreover, we have
\[
\tr(\omega_{R^G})=(fR)^G(R:_{Q(R)}fR)^G.
\]
\end{prop}
\begin{proof} 
 Since $G$ is generated by $\phi=\mathrm{diag}(\lambda^{t_1}, \dots, \lambda^{t_{d}})$, to prove that $f$ is a $G$-canonical element it is enough to observe that $\phi(f)=\lambda^{t_1} \dots \lambda^{t_d} f=\det(\phi)f$. 

We now prove that $\tr(\omega_{R^G})=(fR)^G(R:_{Q(R)}fR)^G$. 
First, notice that by Lemma~\ref{lemma:tracecanonicalideal} and Theorem~\ref{thm:singh-watanabe} we have $\tr(\omega_{R^G})=(fR\cap R^G)(R^G :_{Q(R^G)} (fR \cap R^G))$. So, since $fR\cap R^G=(fR)^G$, it is enough to  prove that
\[
R^G :_{Q(R^G)} (fR \cap R^G)=(R:_{Q(R)}fR)^G.
\]
The inclusion $(R:_{Q(R)}fR)^G \subseteq (R^G :_{Q(R^G)} (fR \cap R^G))$ is clear. Conversely, consider an element $a/b \in (R^G :_{Q(R^G)}(fR \cap R^G))$ with $\gcd(a,b)=1$.
By hypothesis $\gcd(t_1, \dots, t_{d-1},n)=1$ and so there exist $a_1, \dots, a_{d-1}$ positive integers such that $a_1 t_1 + \dots + a_{d-1}t_{d-1} \equiv 1 \mod n$. Therefore, there exists a positive $r$ such that 
\[(a_1 t_1 + \dots + a_{d-1}t_{d-1})r+ t_1+ \dots + t_{d-1}+t_d \equiv 0 \mod n. \]
This implies that $h=(x_1^{a_1}x_2^{a_2}\dots x_{d-1}^{a_{d-1}})^rf \in R^G$, because $\phi(h)=(\prod_{i=1}^{d-1}\lambda^{ra_i t_i}  \prod_{i=1}^{d}\lambda^{t_i})h=h$. It follows that $ah/b \in R^G$, then $b$ is a monomial and $x_d^2$ does not divide $b$. Since we can repeat the same reasoning with respect to every variable, we get that $b$ is squarefree and, therefore, $b$ divides $f$. This means that $af/b \in R$ and, then, $a/b \in (R:_{Q(R)} fR)^G$ as required.
\end{proof}

\begin{lem}\label{lemma-tracecriterion}
Let $h=x_1^{a_1} \dots x_d^{a_d}$ be a monomial of $R^G$ and let $f=x_1 x_2 \dots x_d$. Then, $h\in\tr(\omega_{R^G})$ if and only if one of the following two conditions holds:
\begin{enumerate}
	\item $a_i>0$ for all $i=0,\dots,d$;
	\item $h=x_{\sigma(1)}^{a_{1}}\cdots x_{\sigma(j)}^{a_{j}}$ with $j<d$, where $\sigma$ is a permutation of $\{1,\dots,d\}$ and there exist integers $b_1,\dots,b_j$ such that $0 < b_k \leq a_k+1$ for every $k \in \{1,\dots, j\}$ and  $\sum_{k=1}^j b_k t_{\sigma(k)} \equiv - \sum_{k=j+1}^d t_{\sigma(k)} \mod n$.
\end{enumerate}
\end{lem}
\begin{proof}
If $a_i>0$ for every $i$, we observe that $h \in fR \subseteq \tr(\omega(R^G))$ because $1 \in (R:_{Q(R)}fR)^G$.
Therefore, without loss of generality we suppose that $h=x_1^{a_1} \dots x_j^{a_j}$ for some $j < d$.
We recall that by Proposition~\ref{prop:cyclictrace} $\tr(\omega_{R^G})=(fR)^G(R:_{Q(R)}fR)^G$.
Moreover, we observe that $(fR)^G$ and $(R:_{Q(R)}fR)^G$ are generated by monomials and quotient of monomials respectively  because $f$ is a monomial and $G$ is cyclic. Therefore, since $h$ is a monomial, we have that $h\in\tr(\omega_{R^G})$ if and only if there is an equality
\begin{equation} \label{h}
h=(x_1^{b_1} \dots x_d^{b_d}) \ \frac{x_1^{c_1} \dots x_d^{c_d}}{x_1^{e_1} \dots x_d^{e_d}}
\end{equation}
with $x_1^{b_1} \dots x_d^{b_d} \in (fR)^G$ and $x_1^{c_1} \dots x_d^{c_d}/x_1^{e_1} \dots x_d^{e_d} \in (R:_{Q(R)}fR)^G$, where we assume that the fraction is irreducible.
Since $f=x_1 \dots x_d$, it follows that  $b_i \geq 1 $ and $e_i \leq 1$ for every $i=1, \dots, d$.
Moreover, for every $k=j+1, \dots, d$ we have $a_k=0$ which implies  $b_k=e_k=1$ and $c_k=0$.
We also note that $a_i=b_i+c_i-e_i$ for every $i$ and, therefore, $1 \leq b_i \leq a_i+1$.
Since $h \in R^G$, if $x_1^{b_1} \dots x_j^{b_j} x_{j+1} \dots x_d$ is invariant under the action of $G$, also $x_1^{c_1} \dots x_j^{c_j}/x_1^{e_1} \dots x_j^{e_j} x_{j+1} \dots x_d$ is invariant. 
Recall that $R^G=\kk\llbracket x_1^{\alpha_1}\dots x_d^{\alpha_d} \mid \alpha_1+\dots+ \alpha_d \leq n \text{\ and \ } \alpha_1t_1+\dots +\alpha_dt_d \equiv 0 \mod n\rrbracket$. 
Then, it is possible to write $h$ as in (\ref{h}) if and only if there exist integers $b_1,\dots,b_j$ such that $1 \leq b_i \leq a_i+1$ and $\sum_{i=1}^j b_i t_{i} + \sum_{k=j+1}^d t_{k} \equiv 0 \mod n$.
\end{proof}

Observing that $R^G$ is nearly Gorenstein if and only if the conditions of Lemma~\ref{lemma-tracecriterion} hold for every generator of the maximal ideal of $R^G$ we get the following criterion.

\begin{thm} \label{cyclic nearly Gorenstein}
The ring $R^G$ is nearly Gorenstein if and only if for every $0<i<d$, every permutation $\sigma$ of $\{1,\dots,d\}$ and every $i$-tuple $(a_1, \dots, a_i)$ of positive integers such that $a_1+\dots+a_i \leq n$ and $a_1 t_{\sigma(1)}+\dots +a_i t_{\sigma(i)} \equiv 0 \mod n$, there exist integers $b_1, \dots, b_i$ such that $\sum_{j=1}^i b_j t_{\sigma(j)} \equiv - \sum_{k=i+1}^d t_{\sigma(k)} \mod n$ and $0 < b_j \leq a_j+1$ for every $j \in \{1,\dots, i\}$.
\end{thm}

We want to use the previous theorem to find examples of nearly Gorenstein cyclic quotient singularities.
First, we recall that by Watanabe's Theorem a $\frac{1}{n}(t_1,\dots,t_d)$-singularity is Gorenstein if and only if the acting group $G$ is contained in $\SL(d,\kk)$ which is in turn equivalent to the condition $t_1+\cdots+t_d\equiv 0\mod n$.
For instance, for each dimension $d$ the singularity $\frac{1}{n}(1,\dots,1,t_d)$ with $t_d\equiv-d+1\mod n$ is Gorenstein.
In a similar fashion, we can obtain examples of nearly Gorenstein cyclic quotient singularities in every dimension.

\begin{prop}\label{prop:cyclicNGfamily}
Let $d\geq3$ and $n\geq3$ be integers such that $\gcd(-d+2,n)=1$. Choose an integer $t_d\geq1$ such that $t_d \equiv -d+2 \mod n$. Then, the quotient singularity $\frac{1}{n}(1,\dots,1,t_d)$ is nearly Gorenstein, but not Gorenstein.
\end{prop}

\begin{proof}
	As usual let $R=\kk\llbracket x_1,\dots,x_d\rrbracket$ and consider the group $G$ generated by $\mathrm{diag}(\lambda,\dots,\lambda,\lambda^{t_d})$ for a primitive $n$-th root of unity $\lambda\in\kk$, so that $R^G=\frac{1}{n}(1,\dots,1,t_d)$.
	It is clear that $R^G$ is not Gorenstein, since $1+\cdots+1-d+2=1\not\equiv0\mod n$.
	We prove that $R^G$ is nearly Gorenstein by using Theorem~\ref{cyclic nearly Gorenstein}.
	Consider  $0<i<d$, a permutation $\sigma$ of $\{1,\dots,d\}$ and a $i$-tuple $(a_1, \dots, a_i)$ of positive integers such that $a_1+\dots+a_i \leq n$ and $a_1 t_{\sigma(1)}+\dots +a_i t_{\sigma(i)} \equiv 0 \mod n$.
	If $t_{\sigma(1)}=\cdots=t_{\sigma(i)}=1$, then we have $a_1+\dots+a_i=n$. Therefore, the sum $\sum_{j\leq i}b_j$ for $0<b_j\leq a_j+1$ runs over all possible residues modulo $n$, thus there exist $b_j$'s such that $\sum_{j=1}^i b_j t_{\sigma(j)} \equiv - \sum_{k=i+1}^d t_{\sigma(k)} \mod n$ is satisfied.
	\par Suppose now that $t_{\sigma(1)}=\cdots=t_{\sigma(i-1)}=1$ and $t_{\sigma(i)}=t_d\equiv-d+2\mod n$.
	We distinguish two possibilities.
	If $i=1$, then we have $(-d+2)a_1\equiv0\mod n$ and, since $-d+2$ is coprime with $n$, we obtain $a_1\equiv0\mod n$, which forces $a_1=n$ being $a_1\leq n$.
	Therefore, $0<b_1\leq a_1+1$ ranges over all possible residues modulo $n$ and we conclude as before.
	Suppose now that $i>1$. We choose $b_i=a_i+1$, $b_{i-1}=a_{i-1}$, and $b_{j}=a_j+1$ for all $j=1,\dots,i-2$. Then, we obtain
	\[
	b_1+\dots+b_{i-1}+b_i t_d\equiv a_1+\dots+a_{i-1}+(-d+2)a_i+(i-2)\cdot 1+(-d+2) \equiv -(d-i)\mod n.
	\]
	Hence, $R^G$ is nearly Gorenstein.
\end{proof}

The case $d=2$ was left out from the previous proposition, but in fact two-dimensional cyclic quotient singularities are always nearly Gorenstein.

\begin{cor}\label{cor:cyclicarenearlyG}
	If $R=\kk\llbracket x_1,x_2\rrbracket$ and $G$ is cyclic, then $R^G$ is nearly Gorenstein.
\end{cor}

\begin{proof}
	Let $a_1 \leq n$ be such that $a_1 t_{\sigma(1)} \equiv 0 \mod n$. Since there are no pseudo-reflections in $G$, we have $\gcd(t_{\sigma(1)}, n)=1$ and, then, $n$ divides $a_1$. In particular, $a_1 = n$. Therefore, there is a solution of the equation $b_1 t_{\sigma(1)} \equiv -t_{\sigma(2)} \mod n$ such that $0 < b_1 \leq n=a_1$ and Theorem \ref{cyclic nearly Gorenstein} implies that $R^G$ is nearly Gorenstein.
\end{proof}

Now, we focus on groups with small order. 
We recall that for $n=2$ and $d$ even the ring $R^G$ is Gorenstein since $G\subseteq \SL(d,\kk)$. 
More generally, we prove that for $n\leq3$  it is always nearly Gorenstein.

\begin{cor}\label{cor:nleq3isNG}
	Let $G$ be a cyclic group of order at most $3$, then $R^G$ is nearly Gorenstein.
\end{cor}

\begin{proof}
	We prove only the case of order $3$, since the case when $|G|=2$ can be done in the same way.
	So, assuming $|G|=3$, we will prove that $R^G$ is nearly Gorenstein by using Theorem \ref{cyclic nearly Gorenstein}. Let $(a_1, \dots, a_i)$ be positive integers such that $a_1 t_{\sigma(1)}+ \dots + a_i t_{\sigma(i)} \equiv 0 \mod 3$ for a permutation $\sigma$ of $\{1, \dots, d\}$. We need to find positive integers $b_j \leq a_j+1$ such that $\sum_{j=1}^i b_j t_{\sigma(j)}\equiv -\sum_{k=i+1}^d t_{\sigma(k)} \mod 3$. If $-\sum_{k=i+1}^d t_{\sigma(k)} \equiv 0 \mod 3$, it is enough to set $b_j=a_j$ for every $j=1, \dots, i$. If $-\sum_{k=i+1}^d t_{\sigma(k)} \equiv 1 \mod 3$ and there exists $1 \leq p\leq i$ such that $t_{\sigma(p)}\equiv 1 \mod 3$, then we can set $b_p=a_p+1$ and $b_j=a_j$ for $1 \leq j \leq i$, $j \neq p$. Assume now that $t_{\sigma(j)} \equiv 2 \mod 3$ for every $j=1, \dots, i$. If $i=1$, then $a_1$ has to be equal to $3$ and we can put $b_1=2$, otherwise it is enough to set $b_1=a_1+1$, $b_2=a_2+1$ and $b_j=a_j$ for $3 \leq j \leq i$. 
	The case $-\sum_{k=i+1}^d t_{\sigma(k)} \equiv 2 \mod 3$ is analogous to the previous one.
\end{proof}

As soon as the dimension of $R$ is bigger than $2$ and the order of $G$ is greater than $3$, it is possible to find cyclic quotient singularities $R^G$ that are not nearly Gorenstein. In order to exhibit  some examples we state a necessary condition which follows immediately from Theorem \ref{cyclic nearly Gorenstein}.

\begin{rem}\label{rem:previousremark}
Let $\frac{1}{n}(t_1,\dots,t_d)$ be nearly Gorenstein.
If $\gcd(t_{\sigma(1)}, \dots, t_{\sigma(i)},n)=m>1$ for some $i>0$ and some permutation $\sigma$ of $\{1,\dots,d\}$, then $t_{\sigma(i+1)}+\dots+t_{\sigma(d)} \equiv 0 \mod m$. 
Indeed, if we choose $a_1=n$, Theorem \ref{cyclic nearly Gorenstein} implies that there exists $b_1$ such that $b_1 t_{\sigma(1)} \equiv - \sum_{k=2}^d t_{\sigma(k)} \mod n$. 
Therefore, it is enough to consider this congruence modulo $m$.
\end{rem}

\begin{ex}\label{ex:cyclicnotNG}
(1) Let $\gcd(n,t_1)=m >2$, with $\gcd(m+1,n)=1$ and let $t_1=t_2=\dots = t_{d-2}$ and $t_{d-1}=t_{d}=m+1$. Therefore, $\gcd(t_{1}, \dots, t_{d-2},n)=m>1$, but $t_{d-1}+t_{d} \equiv 2 \not\equiv 0 \mod m$. Hence, $R^G$ is not nearly Gorenstein by the previous remark. For instance, $\frac{1}{8}(4,5,5)= \frac{1}{8}(1,1,4)$ is not nearly Gorenstein. \\
(2) Let $t_1=1, t_2=n-1, t_3=n-2$. We have $t_1+t_2 \equiv 0 \mod n$, but there are no $0 < b_1, b_2 \leq 2$ such that $b_1-b_2 \equiv 2 \mod n$. Hence, Theorem \ref{cyclic nearly Gorenstein} implies that $\frac{1}{n}(t_1,t_2,t_3)$ is not nearly Gorenstein. In particular $\frac{1}{4}(1,2,3)$ is not nearly Gorenstein.
We also notice that in this case, if $n>3$ is odd, we have $\gcd(t_{\sigma(1)}, \dots, t_{\sigma(i)},n)=1$ for every $i$ and, therefore, the converse of Remark~\ref{rem:previousremark} does not hold. 
\end{ex}

Another interesting class of nearly Gorenstein quotient singularities is given by Veronese subalgebras, which are obtained when $t_1=t_2= \dots = t_d=1$. See \cite[Corollary 4.8]{HHS19} for a proof in the positively graded case. 

\begin{cor}\label{cor:VeroneseareNG}
	The Veronese subalgebras of $R$ are nearly Gorenstein.
\end{cor}

\begin{proof}
Let $0<i<d$ and let $a_1, \dots, a_i$ be positive integers such that $a_1+\dots+a_i=n$. Let $kn < d \leq (k+1)n$ for some non-negative integer $k$.
	
	Assume first that $d-i-kn \geq 0$. Then, we have $0 \leq d-i-kn \leq n-i$ that implies $i \leq n-(d-i-kn) \leq n = \sum_{j=1}^i a_j$. Therefore, there exist $b_1, \dots, b_i$ such that $1 \leq b_j \leq a_j$ and $\sum_{j=1}^i b_j t_{\sigma(j)}\equiv n-(d-i-kn) \equiv -\sum_{k=i+1}^d t_{\sigma(k)} \mod n$.
	
	Assume now that $d-i-kn<0$. It follows that $0<kn-d+i<i$ and then \[\sum_{j=1}^{kn-d+i} (a_j+1)t_{\sigma_j}+\sum_{l=kn-d+i+1}^i a_lt_{\sigma_l} \equiv n+kn-d+i \equiv -(d-i) \equiv - \sum_{k=i+1}^d t_{\sigma(k)} \mod n.\]
	Hence, the claim follows by Theorem \ref{cyclic nearly Gorenstein}.
\end{proof}

\par In Table~\ref{Table NG} we present an exhaustive list of  non-isomorphic cyclic quotient singularities for $d=3$ and $4\leq n\leq7$, and for $d=4$ and $4\leq n\leq 6$. Moreover, by using the numerical criterion of Theorem~\ref{cyclic nearly Gorenstein}, we report if they are Gorenstein, nearly Gorenstein or not.

\begin{center}
	\begin{table}[h]
		\setlength\extrarowheight{1mm}
		\begin{tabular}{c|c||c|c||c|c}
			Ring & \ Is nearly Gor. \ & Ring & \ Is nearly Gor. \ & Ring & \ Is nearly Gor. \\
			\hline
			$\frac{1}{4}(1, 1, 1 )$  & NG & 
			$\frac{1}{4}( 1, 1, 2 )$  & G & 
			$\frac{1}{4}( 1, 1, 3 )$  & NG \\
			$\frac{1}{4}( 1, 2, 3 )$  & not NG &
			$\frac{1}{5}( 1, 1, 1 )$ & NG &
			$\frac{1}{5}( 1, 1, 2 )$ & NG \\
			$\frac{1}{5}( 1, 1, 3 )$ & G & 
			$\frac{1}{5}( 1, 1, 4 )$ & NG & 
			$\frac{1}{5}( 1, 2, 3 )$ & not NG \\ 
			$\frac{1}{6}( 1, 1, 1 )$ & NG &
			$\frac{1}{6}( 1, 1, 2 )$ & NG & 
			$\frac{1}{6}( 1, 1, 3 )$ & not NG  \\
			$\frac{1}{6}( 1, 1, 4 )$ & G & 
			$\frac{1}{6}( 1, 1, 5 )$ & NG & 
			$\frac{1}{6}( 1, 2, 3 )$ & G \\
			$\frac{1}{6}( 1, 2, 5 )$ & not NG & 
			$\frac{1}{6}( 1, 3, 4 )$ & not NG & 
			$\frac{1}{6}( 1, 3, 5 )$ & not NG \\
			$\frac{1}{7}( 1, 1, 1 )$ & NG &
			$\frac{1}{7}( 1, 1, 2 )$ & NG &
			$\frac{1}{7}( 1, 1, 3 )$ & not NG \\ 
			$\frac{1}{7}( 1, 1, 4 )$ & NG &
			$\frac{1}{7}( 1, 1, 5 )$ & G & 
			$\frac{1}{7}( 1, 1, 6 )$ & NG \\
			$\frac{1}{7}( 1, 2, 3 )$ & not NG & 
			$\frac{1}{7}( 1, 2, 4 )$ & G & 
			$\frac{1}{7}( 1, 2, 5 )$ & not NG \\
			$\frac{1}{7}( 1, 2, 6 )$ & not NG & 
			$\frac{1}{4}( 1, 1, 1, 1 )$ & G &
			$\frac{1}{4}( 1, 1, 1, 2 )$ & not NG \\
			$\frac{1}{4}( 1, 1, 1, 3 )$ & not NG &
			$\frac{1}{4}( 1, 1, 2, 2 )$ & NG &
			$\frac{1}{4}( 1, 1, 2, 3 )$ & not NG \\
			$\frac{1}{4}( 1, 1, 3, 3 )$ & G  &
			$\frac{1}{4}( 1, 2, 2, 3 )$ & G &
			$\frac{1}{5}( 1, 1, 1, 1 )$ & NG \\
			$\frac{1}{5}( 1, 1, 1, 2 )$ & G &
			$\frac{1}{5}( 1, 1, 1, 3 )$ & NG &
			$\frac{1}{5}( 1, 1, 1, 4 )$ & not NG \\
			$\frac{1}{5}( 1, 1, 2, 2 )$ & NG &
			$\frac{1}{5}( 1, 1, 2, 3 )$ & NG &
			$\frac{1}{5}( 1, 1, 2, 4 )$ & not NG \\
			$\frac{1}{5}( 1, 1, 3, 4 )$ & NG &
			$\frac{1}{5}( 1, 1, 4, 4 )$ & G &
			$\frac{1}{5}( 1, 2, 3, 4 )$ & G \\
			$\frac{1}{6}( 1, 1, 1, 1 )$ & NG &
			$\frac{1}{6}( 1, 1, 1, 2 )$ & not NG &
			$\frac{1}{6}( 1, 1, 1, 3 )$ & G \\
			$\frac{1}{6}( 1, 1, 1, 4 )$ & not NG &
			$\frac{1}{6}( 1, 1, 1, 5 )$ & not NG &
			$\frac{1}{6}( 1, 1, 2, 2 )$ & G \\
			$\frac{1}{6}( 1, 1, 2, 3 )$ & not NG &
			$\frac{1}{6}( 1, 1, 2, 4 )$ & NG &
			$\frac{1}{6}( 1, 1, 2, 5 )$ & not NG \\
			$\frac{1}{6}( 1, 1, 3, 3 )$ & not NG &
			$\frac{1}{6}( 1, 1, 3, 4 )$ & not NG &
			$\frac{1}{6}( 1, 1, 3, 5 )$ & not NG \\
			$\frac{1}{6}( 1, 1, 4, 4 )$ & NG &
			$\frac{1}{6}( 1, 1, 4, 5 )$ & not NG &
			$\frac{1}{6}( 1, 1, 5, 5 )$ & G \\
			$\frac{1}{6}( 1, 2, 2, 3 )$ & not NG &
			$\frac{1}{6}( 1, 2, 2, 5 )$ & not NG &
			$\frac{1}{6}( 1, 2, 3, 3 )$ & not NG \\
			$\frac{1}{6}( 1, 2, 3, 4 )$ & not NG &
			$\frac{1}{6}( 1, 2, 3, 5 )$ & not NG &
			$\frac{1}{6}( 1, 2, 4, 5 )$ & G \\
			$\frac{1}{6}( 1, 3, 3, 4 )$ & not NG &
			$\frac{1}{6}( 1, 3, 3, 5 )$ & G &
			$\frac{1}{6}( 1, 3, 4, 4 )$ & G \\
		\end{tabular} 
		\vspace{6pt}
		\caption{Cyclic quotient singularities for $d=3$ and $4\leq n\leq7$, and for $d=4$ and $4\leq n\leq 6$. G means Gorenstein and NG means nearly Gorenstein.}
		\label{Table NG}
	\end{table}
\end{center}

\par We conclude this section by recalling that the residue of a local ring is a measure of how far is a ring from being nearly Gorenstein. In the next theorem we show that we have cyclic quotient singularities of arbitrarily large residue already in dimension $3$.

\begin{thm}\label{theorem:residuelarge}
Let $R=\kk\llbracket x,y,z \rrbracket$ and let $n$ and $m$ be two coprime positive integers with $n\geq3$ and $m<\lceil\frac{n}{2} \rceil$. Consider the group $G$ generated by $\mathrm{diag}(\lambda, \lambda^{m}, \lambda^{n-1})$, where $\lambda$ is a primitive $n$-th root of unit in $\kk$. Then, $\mathrm{res}(R^G)=m$.
In particular, $R^G$ is nearly Gorenstein if and only if $m=1$.
\end{thm}

\begin{proof}
In order to compute $\mathrm{res}(R^G)=\ell_{R^G} (R^G/\tr(\omega_{R^G}))$ we count how many monomials of the maximal ideal $\mathfrak{m}$ of $R^G$ are not in $\tr(\omega_{R^G})$.
We fix $f=xyz$ and we recall that 
$\tr(\omega_{R^G})=(fR)^G \ (R:_{Q(R)}fR)^G$ by Proposition~\ref{prop:cyclictrace}.
Let $g=x^ay^bz^c\in\mathfrak{m}$.
If $a,b,c>0$, then we can write $g=f x^{a-1}y^{b-1}z^{c-1}\in (fR)^G\subseteq \tr(\omega_{R^G})$.
If $b=c=0$, then $g=x^a\in\mathfrak{m}$ implies that $n$ divides $a$.
Then, the condition of Lemma \ref{lemma-tracecriterion} is satisfied for $g$, since $b_1$ gives all possible residues modulo $n$ for $0<b_1\leq a+1$.
Therefore, $x^a\in\tr(\omega_{R^G})$.
Similarly, one obtains that $y^b, z^c\in\tr(\omega_{R^G})$ because $\mathrm{gcd}(m,n)=\mathrm{gcd}(n-1,n)=1$.
\par It remains to check the monomials of the form $x^ay^b, x^az^c, y^bz^c\in\mathfrak{m}$ with $a,b,c>0$. 
We use again the criterion of Lemma~\ref{lemma-tracecriterion}.
If $y^bz^c\in\mathfrak{m}$, then $mb+(n-1)c\equiv0\mod n$.
It follows that $mb+(n-1)(c+1)\equiv n-1\equiv -1=-t_1\mod n$, therefore $y^bz^c\in\tr(\omega_{R^G})$.
If $x^ay^b\in\mathfrak{m}$, then $a+mb\equiv0\mod n$ implies $(a+1)+mb\equiv 1\equiv -(n-1)= -t_3\mod n$, thus $x^ay^b\in\tr(\omega_{R^G})$.
\par Finally, consider a monomial $x^az^c\in\mathfrak{m}$. 
If $a \geq n$ or $c \geq n$, then $x^ay^b \in \tr(\omega_{R^G})$ because $x^n$ and $y^n$ are in $\tr(\omega_{R^G})$, therefore we may assume $a,c < n$.
Since $x^az^c\in\mathfrak{m}$, we have $a+(n-1)c\equiv 0\mod n$, thus $a\equiv c\mod n$ which implies $a=c$.
By Lemma~\ref{lemma-tracecriterion}, $x^az^c\in \tr(\omega_{R^G})$ if and only if there exist $0<b_1,b_3\leq a+1$ such that $b_1-b_3\equiv -m\mod n$.
We notice that $b_1-b_3\in\{a,a-1,\dots,-a+1,-a\}$ and, so, there are exactly $m-1$ monomials in $\mathfrak{m}$ of the form $x^az^a$ that do not satisfy this criterion: $x^a z^a$ with $1 \leq a \leq m-1$. Hence, $\dim_{R^G/\mathfrak{m}} (R^G/\tr(\omega_{R^G}))=\dim_{R^G/\mathfrak{m}} (\mathfrak{m}/\tr(\omega_{R^G}))+1=m$. 
\end{proof}

\section{An addendum to Ding's classification in dimension two}\label{section:dimension2}
Let $\mathbb{C}$ be the field of complex numbers and set $R=\mathbb{C}\llbracket u,v\rrbracket$.
In his paper \cite{Din93}, Ding classified nearly Gorenstein quotient singularities $R^G$, where $G$ is a finite small subgroup of $\GL(2,\mathbb{C})$. 
His result relies on the well-known classification of such subgroups which goes back to Klein \cite{Kle84} (see also \cite{Bea10, Bri67, Rie77}).
However, a nearly Gorenstein quotient singularity was left out of his classification.
We are going to describe it.

\par We consider the octahedral group $\mathbb{O}_{11}$ generated by matrices
\begin{align*}
\phi&=\begin{pmatrix}
\lambda & 0 \\ 0 & \lambda
\end{pmatrix}, \ 
\psi=\begin{pmatrix}
\zeta & 0 \\ 0 & \zeta^{-1}
\end{pmatrix}, \
\tau=\begin{pmatrix}
0 & i \\ i & 0
\end{pmatrix},\
\eta=\frac{1}{\sqrt2}\begin{pmatrix}
\zeta & \zeta^3 \\ \zeta & \zeta^7
\end{pmatrix},
\end{align*}
where $\lambda$ and $\zeta$ are primitive roots of unity in $\mathbb{C}$ of orders $22$ and $8$ respectively.
In other words, $\mathbb{O}_{11}$ is the extension of the binary octahedral subgroup of $\SL(2,\mathbb{C})$ of order $48$ generated by $\psi, \tau, \eta$ with the cyclic group of order $22$ generated by $\phi$.  

\begin{prop}\label{prop:O11isNG}
	The quotient singularity $R^{\mathbb{O}_{11}}$ is nearly Gorenstein.
\end{prop}

\begin{proof}
We consider the polynomials
\[\begin{split}
 g_1&=(u^5v-uv^5)^2, \\
 g_2&=uv(u^4-v^4)(u^{12}-33u^8v^4-33u^4v^8+v^{12}),\\
 g_3&=u^8+14u^4v^4+v^8,
\end{split}
\]
which are generating invariants for the Gorenstein singularity $E_7$ (see e.g. \cite[\S6.16]{LW12}) and in particular invariants for the action of $\psi$, $\tau$, and $\eta$.
By \cite[Satz~6]{Rie77}, a minimal set of generators for the maximal ideal of $R^{\mathbb{O}_{11}}$ is given by
\[
z_1=g_3^{11}, \ z_2=g_1g_3^4, \ z_3=g_1^4g_2, \ z_4=g_2g_3^6, \ z_5=g_1^3g_3, \ z_6=g_1^2g_2g_3^3.
\]
Now, we consider the polynomial $f=g_3^3$. 
It is invariant for the action of the matrices $\psi, \tau, \eta\in\SL(2,\mathbb{C})$, and $\phi(f)=\lambda^{24}f=\det(\phi)f$. 
Therefore, $f$ is a $\mathbb{O}_{11}$-canonical element and, thus, $\omega_{R^{\mathbb{O}_{11}}}=(fR)^{\mathbb{O}_{11}}$ is a canonical ideal of $R^{\mathbb{O}_{11}}$ by Theorem~\ref{thm:singh-watanabe}.
\par We have that $z_1, z_2, z_4$ and $z_6$ are multiples of $f$, then they are in $\tr(\omega_{R^{\mathbb{O}_{11}}})$.
Moreover, the fact 
\[
\frac{z_3}{z_6}=\frac{z_5}{z_2}=\frac{g_1^2}{g_3^3}\in (R^{\mathbb{O}_{11}}:_{Q(R^{\mathbb{O}_{11}})}(fR\cap R^{\mathbb{O}_{11}}))
\]
implies that $z_3=z_6\frac{z_3}{z_6}, \ z_5=z_2\frac{z_5}{z_2}\in\tr(\omega_{R^{\mathbb{O}_{11}}})$ as well.
Hence, $R^{\mathbb{O}_{11}}$ is nearly Gorenstein.
\end{proof}

We also point out that Ding erroneously included cyclic quotient singularities $\frac{1}{n}(1,n-1)$ in his list of nearly Gorenstein not Gorenstein quotient singularities. Indeed, it is well known that $\frac{1}{n}(1,n-1)$  is Gorenstein.

\end{document}